\documentclass{article}

\usepackage{amsmath, amssymb, latexsym, amsthm, xspace, color}

\usepackage{ifpdf}
\usepackage{paralist}
\usepackage{graphicx}
\usepackage{subfigure}
\usepackage[all]{xy}
\usepackage{psfrag}
\ifpdf
\else
\usepackage{epsfig}
\fi

\numberwithin{equation}{section}

\newcommand{\cx}{{\mathbb C}} 
 
\newcommand{\HH}{{\mathbb H}} 
\newcommand{\integers}{{\mathbb Z}}
\newcommand{\nats}{{\mathbb N}}
\newcommand{\ratls}{{\mathbb Q}} 
\newcommand{\reals}{{\mathbb R}} 
\newcommand{\proj}{{\mathbb P}}
\newcommand{\zed}{\integers}

\newcommand{\om}{{\eta}}
\newcommand{\SLOD}{{\rm SL}(\ord_D \oplus \ord_D^\vee)}
 
\def\zz{\mathbf z}  \def\vv{\mathbf v} \def\uu{\mathbf u} \def\xx{\mathbf x}
  \newcommand{\e}{{\bf e}}

\def\be{\begin{equation}}   \def\ee{\end{equation}}     \def\bes{\begin{equation*}}    \def\ees{\end{equation*}}
\def\ba{\be\begin{aligned}} \def\ea{\end{aligned}\ee}   \def\bas{\bes\begin{aligned}}  \def\eas{\end{aligned}\ees}
\def\mat#1#2#3#4{\begin{pmatrix}#1&#2\\#3&#4\\ \end{pmatrix}}  
\newcommand{\ol}{\overline}

\newcommand{\Teichmuller}{Teich\-m\"uller\xspace}

\newcommand{\FF}[1]{\mathbb{F}_{#1}}

\renewcommand{\tilde}{\widetilde}

\newcommand{\moduli}[1][g]{{\mathcal M}_{#1}}
\newcommand{\AVmoduli}[1][g]{{\mathcal A}_{#1}}
\newcommand{\omoduli}[1][g]{{\Omega\mathcal M}_{#1}}

\newcommand{\pomoduli}[1][g]{{\proj\Omega\mathcal M}_{#1}}

\newcommand{\SL}{{\mathrm{SL}}}

\newcommand{\Sp}{\mathrm{Sp}}
\newcommand{\teich}{\mathcal{T}}

\newcommand{\fra}{\mathfrak{a}}
\newcommand{\frb}{\mathfrak{b}}

\DeclareMathOperator{\Jac}{Jac}
\DeclareMathOperator{\Pic}{Pic}
\DeclareMathOperator{\Prym}{Prym}

\DeclareMathOperator{\Tr}{Tr}

\DeclareMathOperator{\diag}{diag}

\newtheoremstyle{example}{3pt}{3pt}{\upshape}{}{\itshape}{.}{.5em}{}

\newtheorem{theorem}{Theorem}[section] 
\newtheorem{prop}[theorem]{Proposition} 

\newtheorem{lemma}[theorem]{Lemma}

\newtheorem{example}[theorem]{Example}
\newtheorem{problem}[theorem]{Problem}

\theoremstyle{example}

\theoremstyle{definition}

\theoremstyle{remark}

\makeatletter
\def\blfootnote{\xdef\@thefnmark{}\@footnotetext}

\renewcommand{\l@section}{\@dottedtocline{0}{1.5em}{2.3em}}
\renewcommand{\l@subsection}{\@dottedtocline{1}{3.8em}{3.2em}}
\renewcommand{\l@subsubsection}{\@dottedtocline{2}{7.0em}{4.1em}}

\makeatother

\newcommand{\GL}{{\rm GL}}
\newcommand{\ord}{\mathfrak{o}}

\newcommand{\cL}{\mathcal{L}}
\newcommand{\cF}{\mathcal{F}}
\newcommand{\cX}{\mathcal{X}}
\newcommand{\cO}{\mathcal{O}}

\newcommand{\legendre}[2]{\left(\frac{#1}{#2}\right)}

\begin{document}

\bibliographystyle{halpha}

\title{Prym covers, theta functions and Kobayashi curves in Hilbert modular surfaces}
\author{Martin M\"oller}

\maketitle

\begin{abstract}
Algebraic curves in Hilbert modular surfaces 
that are totally geodesic for the Kobayashi metric have
very interesting geometric and arithmetic properties, e.g.\ they are rigid.
There are very few methods known to construct such algebraic geodesics that
we call Kobayashi curves.
\par
We give an explicit way of constructing Kobayashi curves using 
determinants of derivatives of theta functions.
This construction also allows to calculate the Euler characteristics
of the \Teichmuller curves constructed by McMullen using Prym covers.
\end{abstract}

\tableofcontents
\section*{Introduction}
\label{sec:introduction}

We call {\em Kobayashi curves}  algebraic curves on Hilbert modular surfaces  
that are totally geodesic for the Kobayashi metric. They are rigid and
interesting both from geometric and arithmetic point of view and there are
very few methods to construct Kobayashi curves. This paper
can be read from two perspectives, the construction of  Kobayashi curves
and the calculation of invariants of \Teichmuller curves.
\par
The most obvious Kobayashi curve on a Hilbert modular surface $X_D = \HH^2/\SLOD$
is the image of the diagonal in $\HH^2$. One can twist it by a matrix $M \in \GL^+_2(\ratls(\sqrt{D})$, 
i.e.\ consider the image of $z \mapsto (Mz, M^\sigma z)$ and obtain further 
Kobayashi curves, also known as Hirzebruch-Zagier cycles or Shimura curves.
These curves are more special, metrically they are even geodesic for the
invariant Riemannian metric on $\HH^2$. The first Kobayashi curves without 
this supplementary property, were constructed implicitly in 
\cite{calta} and \cite{mcmullenbild}. They were constructed as \Teichmuller
curves $W_D \in \moduli[2]$, the name refers to their construction
using Weierstra\ss\ points. Their image under the Torelli map lies in $X_D$
and in this sense the curves $W_D$  generalize the modular embeddings
of triangle groups in \cite{CoWo90} that cover some small discriminants $D$. 
One can apply a twist by a M\"obius transformation $M\in \GL^+_2(\ratls(\sqrt{D}))$ 
also to $W_D$ to obtain more Kobayashi curves. The geometry of the 
resulting curves is studied in \cite{weiss}. 
\par
Not all the Kobayashi curves in $X_D$ arise as twists of $W_D$ or of the diagonal.
In fact, it is shown in \cite{weiss} that an invariant
(second Lyapunov exponent) of a Kobayashi curve is unchanged under twisting.
Moreover, the image of other \Teichmuller curves $W_D(6)$ constructed 
in \cite{mcmullenprym} map to curves  $W_D^X$ in $X_D$ that are
also Kobayashi curves and 
that have a second Lyapunov exponent different from the one of the diagonal 
and the second Lyapunov exponent of $W_D$. 
\par
The first aim of this paper is to construct explicitly modular forms whose
vanishing loci are these Kobayashi curves  $W_D^X$. This part is a continuation
of \cite{moelzag}. There, a theta function interpretation
of the first series of \Teichmuller curves $W_D$ has been found. Having an explicit modular
form at hand can be used to determine the period map explicitly as power series
(whereas this is a great mystery from the \Teichmuller curve perspective) and it
can be used to retrieve all combinatorial properties of $W_D$ (e.g.\ the set of cusps, etc.) 
from an arithmetic perspective
without using the  geometry of flat surfaces. 
\par
To state the first result, let $\theta_0,\theta_1,\theta_2,\theta_3$ be the classical theta constants 
with characteristic $(c_1,0)$ for $c_1 \in \frac12 \zed^2/\zed^2$. The precise definition 
is given in Section~\ref{sec:HMS}. We write 
$f'$ for the derivative of $f$
in the direction $z_2$, where $(z_1,z_2)$ are the coordinates in $\HH^2$.
\par
\begin{theorem} \label{thm:introGDX}
The determinant of derivatives of theta functions
$$G^X_D(\zz) = 
 \left| \begin{matrix} 
\theta_{0}(\zz) & \theta_{1}(\zz) & \theta_{2}(\zz) & \theta_{3}(\zz)\\
\theta'_{0}(\zz) & \theta'_{1}(\zz) & \theta'_{2}(\zz) & \theta'_{3}(\zz)\\
\theta''_{0}(\zz) & \theta''_{1}(\zz) & \theta''_{2}(\zz) & \theta''_{3}(\zz) \\
\theta'''_{0}(\zz) & \theta'''_{1}(\zz) & \theta'''_{2}(\zz) & \theta'''_{3}(\zz) \\
\end{matrix} \right|
$$
is a modular form of weight $(2,14)$ for $\SLOD$. Its vanishing locus is the Kobayashi curve $W_D^X$.
\end{theorem}
\par
Determining invariants of \Teichmuller curves is the motivation to a variant
of this construction for non-principally polarized abelian varieties. From that point of 
view this paper is jointly with the work of \cite{manhlann} a continuation of \cite{mcmullenprym}.
\par
Let $W_D(6)$ be the Prym \Teichmuller curves in $\moduli[4]$ and let $W_D(4)$ be the
Prym \Teichmuller curves in $\moduli[3]$. The notation refers to their construction
using holomorphic one-forms with a $6$-fold resp.\ with a $4$-fold zero. 
(See Section~\ref{sec:HMS} for the
definitions and the construction of these curves via flat surfaces.) 
\par
\begin{theorem} \label{thm:intoXS}
For the Prym \Teichmuller curves $W_D(6) \subset \omoduli[4]$
the Euler characteristic is given by
$$ \chi(W_D(6)) = -7 \chi(X_D).$$
For genus three and $D \equiv 5 (8)$ the locus $W_D(4)$ is empty.
For $D \equiv 4 (8)$ we have 
$$\chi(W_D(4)) = -\frac52 \chi(X_{D,(1,2)})$$
and for $D \equiv 1(8)$ there are two components $W_D(4)^1$ and 
$W_D(4)^2$ with
$$\chi(W_D(4)^j) = -\frac52 \chi(X_{D,(1,2)}) \quad \text{for}\quad j=1,2.$$
\end{theorem}
\par
Here $X_{D,(1,2)}$ is the locus in the moduli space of $(1,2)$-polarized abelian surfaces 
parameterizing surfaces with real multiplication, a Hilbert modular surface
for some Hilbert modular group commensurable to the standard Hilbert modular groups.
We give a precise definition and a way to evaluate explicitly the Euler characteristic in 
Section~\ref{sec:HMS}. The preceding theorem does not prove that $W_D(4)$ nor $W_D(4)^{j}$
is irreducible. This important result is shown in \cite{manhlann}. Connectedness of $W_D(6)$
is conjectured in \cite{manhlann} with  evidence given by small discriminants.
\par
Pictures of the flat surfaces generating the \Teichmuller curves
$W_D$, $W_D(6)$, $W_D(4)$ are drawn in Figure~\ref{cap:WDLSX}.
Note that there is presently no algorithm known to compute the group uniformizing 
the curves $W_D$ directly, i.e.\  using the geometry of
the generating flat surfaces, if $D$ is larger than some small explicit
constant and thus $W_D$ not a rational curve. The same statement holds for
$W_D(6)$ and for $W_D(4)$.
\par
\begin{figure}[h]
\begin{center}
\psfrag{1}{$1$}
\psfrag{2}{$2$}
\psfrag{3}{$3$}
\psfrag{4}{$4$}
\psfrag{5}{$5$}
\psfrag{6}{$6$}
\psfrag{7}{$7$}
\psfrag{8}{$8$}
\epsfig{figure=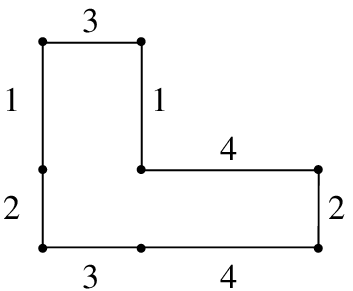, width=3cm}\quad\quad
\epsfig{figure=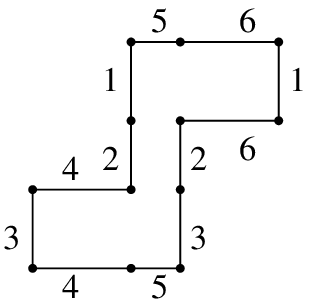, width=3cm}\quad\quad
\epsfig{figure=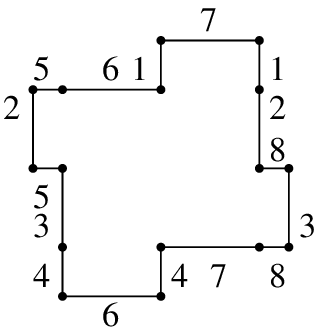, width=3cm}
\end{center}
\caption{Flat surfaces generating respectively the \Teichmuller curves $W_D$, $W_D(6)$ and $W_D(4)$ 
(from \cite{mcmullenprym}). }
\label{cap:WDLSX}
\end{figure}
\par
In order to calculate $\chi(W_D(4))$ we construct Hilbert modular forms for 
the Hilbert modular groups associated with $X_{D,(1,2)}$ whose
vanishing loci are the image curves $W_D^S$ of $\chi(W_D(4))$ in $X_D$.
As for $G_D^X$, these Hilbert modular 
forms $G_D^S$ are very 'canonical', determinants of derivatives
of theta functions. Their precise from is stated in Proposition~\ref{prop:GDS}. The
Euler characteristic of $W_D(6)$ is evaluated, too, using modular forms, without
ever referring to the geometry of flat surfaces. This strategy was first carried out
by Bainbridge (\cite{bainbridge07}) to compute the Euler characteristic of
the curves $W_D$. There, on the contrary, the modular form cutting out $W_D$ was first 
described using relative periods, i.e.\ flat surface geometry.
\par
The construction raises the question to construct more, even to determine all Kobayashi
curves on Hilbert modular surfaces. The construction of modular forms 
using determinants of derivatives of theta functions has a natural analog using
higher multiples of the principal polarization and higher order derivatives. Yet, 
showing that these modular forms define Kobayashi curves requires techniques different
from the ones used here. Some speculations in this direction are the content of
Section~\ref{sec:general}.
\par
The core of Theorem~\ref{thm:introGDX} is Proposition~\ref{prop:GD} as well as some converse statement
derived in Section~\ref{sec:eulerchar}. While proving this converse statement
it turns out that the maps $W_D(6) \to W_D^X$ and $W_D(4) \to W_D^S$ are bijections, 
hence that introducing two names for these curves served for technical purposes only.
The proofs of both Theorem~\ref{thm:introGDX} 
and Theorem~\ref{thm:intoXS} appear at the end of Section~\ref{sec:eulerchar}.
\par
The author thanks Olivier Debarre, Sam Grushevsky, Erwan Lanneau, Curt McMullen 
and Don Zagier for discussions and useful comments.
\par
\section{Background}
\label{sec:HMS}

In this section we develop the notions of real multiplication, Hilbert modular
surfaces and their embedding into the moduli space of principally polarized
abelian varieties in some detail. This is well-known, but the variant
for $(1,2)$-polarizations, that we also need, is treated in less detail in
the literature. We end with some generalities on Kobayashi curves.

\paragraph{Hilbert modular surfaces.}
Let $\ord=\ord_D$ be the order of discriminant $D$ in the
quadratic field $K=\ratls(\sqrt{D})$ with $\sigma$ a generator of the Galois group
of $K/\ratls$. We fix once and for all two
embeddings $\iota, \iota_2: K \to \reals$ and implicitly use the first embedding
unless stated differently. We denote by $X_D$ the Hilbert modular surface of
discriminant $D$, i.e.\ $X_D = \HH^2/\SLOD$. These Hilbert modular surfaces parameterize principally
polarized abelian varieties with real multiplication by $\ord_D$. We give the
details to introduce two types of coordinates since we will soon also need the non-principally
polarized version.
\par
To a point $\zz = (z_1,z_2) \in \HH^2$ we associate the polarized abelian variety 
$A_\zz = \cx^2/\Lambda_\zz$ where $\Lambda_\zz$ is the lattice
\be \label{eq:lazz}
 \Lambda_\zz = \{(a + bz_1, a^\sigma+b^\sigma z_2)^T \,|\,\, a \in \ord_D, b \in \ord_D^\vee \} 
\ee
We denote the coordinates of $\cx^2$ by $\uu = (u_1,u_2)^T$ and we see that
that real multiplication is given for $\lambda \in \ord_D$ by 
$\lambda \cdot(u_1,u_2)^T = (\lambda u_1, \lambda^\sigma u_2)^T$. Consequently, 
the holomorphic one-forms $du_1$ and $du_2$ on $A_\zz$ are eigenforms for
 real multiplication and 
we refer to $\uu$ as {\em eigenform coordinates}. We also say that $du_1$ is the  eigenform 
for $\iota$, unique up to scalar, i.e.\ where $\lambda \in K$ acts by $\lambda du_1 =
\iota(\lambda) \cdot du_1$. Note that here and in the sequel we represent the 
universal covering $\cx^2$ of $A_\zz$ by column vectors. 
\par
The pairing 
\be \label{eq:tracep}
\langle(a,b),(\tilde{a},\tilde{b})  \rangle = \Tr^K_\ratls(a\tilde{b} - \tilde{a}b) 
\ee
is integer valued on $\ord_D \oplus \ord_D^\vee$ (which we identify with $\Lambda_\zz)$ and 
unimodular by definition of $\ord_D^\vee$. It thus defines a principal polarization on $A_\zz$.
\par
Associated with any choice of $\zed$-basis $(\om_1,\om_2)$ of $\ord_D$ 
satisfying the sign convention $ \Tr_\ratls^K(\om_1\om_2^\sigma) = +\sqrt{D}$
there is a symplectic basis  of $\ord_D \oplus \ord_D^\vee$ given by
$$a_1 = (\om_1,0)^T, a_2 = (\om_2,0)^T, b_1=(0,\om_2^\sigma/\sqrt{D} ) ^T, b_2 = (0,\om_1^\sigma/\sqrt{D} )^T. $$
In this basis of homology and in the eigenform coordinates, the period matrix of $A_\zz$
is given by 
$$\Pi_\uu = \left(\begin{matrix} \om_1 & \om_2 & \om_2^\sigma z_1/\sqrt{D} & -\om_1^\sigma z_2/\sqrt{D} 
\\ \om_1^\sigma & \om_2^\sigma & -\om_2 z_1/\sqrt{D} & \om_1z_2/\sqrt{D} \\
\end{matrix} \right) = \left(\begin{matrix} B\,\, & \mat {z_1}00 {z_2} A^T \\
\end{matrix} \right),$$
where 
\be \label{eq:SMEA}  A=B^{-1}\,,\qquad B = \mat{\om_1}{\om_2}{\om_1^\sigma}{\om_2^\sigma} \;.\ee 
The change of basis $\vv = (v_1,v_2)^T = A\cdot \uu$ results in multiplying $\Pi_\uu$ 
from the left by $A$ and gives
\be \label{eq:Pivv} \Pi_\vv = \left(\begin{matrix} I_2\,\, & A\mat {z_1}00 {z_2} A^T \\ \end{matrix} \right) \ee  
Since the period matrix of $A_\zz$ is in standard form with respect to the
basis $\vv = (v_1,v_2)$ (and a $\zed$-basis of homology) we refer to $\vv$ as {\em standard coordinates}.
\par

\paragraph{Siegel modular embeddings.} Any choice of a $\zed$-basis  $(\om_1,\om_2)$ of $\ord_D$, 
defines a map $\psi: \HH^2 \to \HH_2$ that is equivariant with respect to a 
homomorphism $\Psi:\SLOD \to\Sp(4,\zed)$ and descends to a map $X_D \to \AVmoduli[2]$. This map
is the inclusion of the locus of real multiplication into the moduli space of abelian surfaces.
Explicitly, from \eqref{eq:Pivv} we know that 
\be \label{eq:SME} \psi\,: \quad \zz\,=\,(z_1,z_2)\;\mapsto\; A\mat {z_1}00 {z_2} A^T \ee
where $A$ and $B$ are defined in \eqref{eq:SMEA}. 
If we let
\be \label{eq:SMEG}   \Psi\,: \quad  \mat abcd\mapsto\mat A00{B^T}\mat{\widehat a}{\widehat b}{\widehat c}{\widehat d}\mat B00{A^T}\,,\ee
where $\widehat a$ for $a\in K$ denotes the diagonal matrix $\,\diag(a,\sigma(a))$, then the equi\-variance
is easily checked. 


\paragraph{Abelian surfaces with real multiplication and a $(1,2)$-polarization.}

We define $X_{D,(1,2)}$ to be the moduli space of $(1,2)$-polarized abelian varieties with
real multiplication by $\ord_D$.
\par
\begin{prop}
The locus  $X_{D,(1,2)}$ is empty for $D \equiv 5\,\, (8)$. The locus  $X_{D,(1,2)}$ 
is non-empty and irreducible both for  $D \equiv 0,4\,\, (8)$ and 
for $D \equiv 1\,\, (8)$.
\end{prop}
\par
\begin{proof}
Suppose that $(A = \cx^2/\Lambda,\cL)$ is an abelian variety with real multiplication
and a $(1,2)$-polarization $\cL$. Then $\Lambda$ is a rank-two $\ord_D$-module with
symplectic pairing of signature $(1,2)$.
By \cite{Bass62} such a lattice splits as a direct sum of $\ord_D$-modules. Moreover, 
although $\ord_D$ is not a Dedekind domain for $D$ a non-fundamental discriminant, 
any rank-two $\ord_D$-module is isomorphic to $\fra \oplus \ord_D^\vee$ for some
fractional $\ord_D$-ideal $\fra$. The isomorphism can moreover be chosen so that
the symplectic form is mapped to the trace pairing \eqref{eq:tracep}. In this normalization, 
if $N^F_\ratls(\fra) = h$, then the polarization has degree $h^2$. Since all polarizations 
of degree four are of type $(1,2)$, the locus  $X_{D,(1,2)}$ is non-empty if and only
if there is a fractional $\ord_D$ ideal $\fra$ with $N^F_\ratls(\fra) = 2$, 
i.e. if and only if  $D \not \equiv 5\,\, (8)$.
\par
Generalizing \eqref{eq:lazz} we define for any ideal $\frb$ and $\zz \in \HH^2$ the lattice
\be \label{eq:labzz}
 \Lambda^\frb_\zz = \{(a + bz_1, a^\sigma+b^\sigma z_2)^T \,|\,\, a \in \frb, b \in \ord^\vee \}.
\ee
For $D \equiv 0,4\,\, (8)$ there is exactly one prime ideal $\fra$ of norm two, so
as in the principally polarized case one shows that $X_{D,(1,2)} = \HH^2/\SL_2(\fra \oplus \ord_D^\vee)$
is connected.
\par
For $D \equiv 1\,\, (8)$ the prime two splits $(2) = \fra \fra^\sigma$ into two prime ideals
of norm two. Consequently, both $\HH^2/\SL_2(\fra \oplus \ord_D^\vee)$ and 
$\HH^2/\SL_2(\fra^\sigma \oplus \ord_D^\vee)$  parameterize  $(1,2)$-polarized abelian varieties with
real multiplication by $\ord_D$ given by $\cx^2/ \Lambda^\frb_\zz$ and $\cx^2/ \Lambda^{\frb^\sigma}_\zz$.
Since $\ord_D^\vee = (\ord_D^\vee)^\sigma$,  the map $(u_1,u_2) \mapsto (u_2,u_1)$ defines for any fractional
ideal $\frb$ an isomorphism 
\be \label{eq:isotau}
\cx^2/ \Lambda^\frb_{(z_1,z_2)} \to \cx^2/ \Lambda^{\frb^\sigma}_{(z_2,z_1)}
\ee
of abelian varieties. Consequently, $\HH^2/\SL_2(\fra \oplus \ord_D^\vee)$ and 
$\HH^2/\SL_2(\fra^\sigma \oplus \ord_D^\vee)$ parameterized the same subset of the 
moduli space of $(1,2)$-polarized abelian varieties.
\end{proof}
\par

\paragraph{Period matrices.} 

We choose from now on a symplectic basis of $(\om_1,\om_2)$ of $\ord_D$, 
such that $(\om_1, 2\om_2)$ is a basis of $\fra$ and such that the sign convention
$ \Tr^K(\om_1\om_2^\sigma) = +\sqrt{D}$ holds. Let $P=\mat1002$ be the diagonal 
matrix of the type of the polarization
we are interested in. With this choice of generators, the basis 
$$a_1 = (\om_1,0)^T, a_2 = (2\om_2,0)^T, b_1=(0,\om_2^\sigma/\sqrt{D} ) ^T, b_2 = (0,\om_1^\sigma/\sqrt{D} )^T. $$
is in standard form with respect to the trace polarization. In this basis the period matrix is
$$\Pi_\uu = \left(\begin{matrix} \om_1 & 2\om_2 & \om_2^\sigma z_1/\sqrt{D} & -\om_1^\sigma z_2/\sqrt{D} 
\\ \om_1^\sigma & 2\om_2^\sigma & -\om_2 z_1/\sqrt{D} & \om_1z_2/\sqrt{D} \\
\end{matrix} \right) = \left(\begin{matrix} B & \mat {z_1}00 {z_2} A^T \\
\end{matrix} \right),$$
where 
\be \label{eq:SMEAP}  A=PB^{-1}\,,\qquad B = \mat{\om_1}{2\om_2}{\om_1^\sigma}{2\om_2^\sigma} \;.\ee 
The change of basis $\vv = (v_1,v_2)^T = A\cdot \uu$ results in multiplying $\Pi_\uu$ 
from the left by $A$ and gives
\be \label{eq:Pivv12} \Pi_\vv = \left(\begin{matrix} P & A\mat {z_1}00 {z_2} A^T \\ \end{matrix} \right), \ee  
the period matrix for a $(1,2)$-polarized abelian surface with real multiplication in standard coordinates.

\paragraph{Euler characteristics.} The notion Euler characteristic (of curves and of Hilbert modular surfaces)
refers throughout to orbifold Euler characteristics. Let  $D=f^2 D_0$ be the factorization of the discriminant into a 
fundamental discriminant $D_0$ and a square of $f \in \nats$. A reference to compute 
the Euler characteristic of the Hilbert modular surfaces $X_D$,
including the case of non-fundamental
discriminants, is \cite[Theorem~2.12]{bainbridge07}. His formula is
\be \label{eq:chiXD}
\chi(X_D) = 2f^3 \zeta_{\ratls(\sqrt{D})}(-1) \left(\sum_{r|f} \legendre{D_0}{r} \frac{\mu(r)}{r^2} \right),
\ee
where $\mu$ is the M\"obius function and $\legendre{a}{b}$ is the Jacobi symbol.
\par 
The groups $\SLOD$ and $\SL_2(\fra \oplus \ord_D^\vee)$ are commensurable. To determine
the indices in their intersection, 
we conjugate both groups by $\mat {\sqrt{D}} 0 0 1$. This takes $\SLOD$ into $\SL(\ord_D \oplus \ord_D)$, 
and  $\SL_2(\fra \oplus \ord_D^\vee)$ into  $\SL_2(\fra \oplus \ord_D)$. The two images under conjugation
contain
\be
\Gamma_\fra = \left\{ \mat abcd \in \SL_2(K) \,:\, a,d \in \ord_D, \,
b \in \fra,\, c \in \ord_D \right\}.
\ee
with a finite index that we now calculate. We  reduce mod $\fra$. Since 
$\ord_D/\fra \cong \FF2$, the group $\SL(\ord_D \oplus \ord_D)$
reduces to the full group $\SL_2(\FF2)$ and $\Gamma_\fra$ reduces to the group of lower
triangular matrices. Thus 
$[\SLOD : \Gamma_\fra] = 3.$
\par
Suppose the fractional ideal $\fra$ is not invertible, i.e.\ $2 | f$, hence
$D/4$ is also the discriminant of a ring. We claim that $\fra^{-1} = \ord_{D/4}$.
In fact, $\fra = \langle 2,x \rangle_{\zed}$ for some $x$ of norm $4$. Then 
$\fra^{-1} = \langle 1, x^\sigma/2 \rangle_{\zed} = \ord_{D/4}$.
Using the claim we consider $\Gamma_\fra$ and $\SL_2(\fra \oplus \ord_D)$ as
subgroups of $\SL(\ord_{D/4} \oplus \ord_{D/4})$. Both contain the kernel of
the reduction mod $\fra$ to $\SL_2(\ord_{D/4}/\fra)$. The images are groups of lower triangular
matrices of size $2$ and $4$ respectively. We conclude that 
$[\SL(\fra \oplus \ord_D) : \Gamma_\fra] = 2$ in this case.
\par
Suppose that $\fra$ is invertible. If $\fra$ is a principal ideal, generated by $\lambda$, 
then conjugation by ${\rm diag}(\lambda,1)$ takes  $\SL_2(\fra \oplus \ord_D)$ into
 $\SL(\ord_D \oplus \ord_D)$ and $\Gamma_\fra$ into the transposed group. Hence
$[\SL(\fra \oplus \ord_D) : \Gamma_\fra] = 3$ in this case. The general case with
$\fra$ invertible behaves similarly. Outside the prime(s) lying over the ideal $(2)$ the
groups $\SL(\fra \oplus \ord_D)$ and $\Gamma_\fra$ agree. We localize at $\fra$
and take the completion. Since $2\!\not| f$, the ideal $\fra$ now becomes
principal and by the preceding argument $[\SL(\fra \oplus \ord_D) : \Gamma_\fra]$ divides $3$.
Since the two groups are not the same, equality holds. Altogether we have
shown the following proposition.
\par
\begin{prop} \label{prop:vol12}
The Euler characteristic of $X_{D,(1,2)}$ and of $X_D$ are related as follows.
\be
\frac{\chi(X_{D,(1,2)})}{\chi(X_D)} = \left\{\begin{array}{lcl}
1 & \text{if} & 2\! \not | \,f \\
3/2 & \text{if} & 2 \,\,| \, f. \\
\end{array} \right.
\ee
\end{prop}
\par

\paragraph{Theta functions.}
Let $(A,\cL)$ be a $P$-polarized abelian surface, where $P$ is a diagonal matrix, 
the type of the polarization. If we fix a basis of homology so that the
polarization is in standard form $\mat0P{-P}0$, then 
 $A = \cx^2/\Lambda$, where $\Lambda = (\begin{matrix} P & Z\end{matrix}) \cdot \zed^4$
and $Z \in \HH_2$.  
The classical theta functions with characteristic $(c_1,c_2)$ are given (using
standard coordinates)  by
\be \label{eq:thclass}
\theta_{c_1,c_2}(Z,\vv) = \sum_{\xx \in \zed^2} \e(\pi i (\xx+c_1)^TZ(\xx+c_1) 
+ 2\pi i(\xx+c_1)^T(\vv+c_2)). \ee
The main result we need is that for $(c_1,c_2)$ fixed, the
set 
\be \label{eq:thbasis}
 \left\{ \theta_{c_1 + m_1,c_2}(Z,\vv),  \quad m_1 \in P^{-1}\zed^2/\zed^2\right\}  \ee 
forms a basis of a translate of $\cL$ by the point $c = Zc_1 + c_2$, see \cite[Section~3]{bl}).
\par
If one wants to work out explicitly the modular forms $G_D^X$ defining $W_D^X$
(see Theorem~\ref{thm:introGDX}) and the corresponding modular form for
the genus three construction (see Proposition~\ref{prop:GDS}) 
one has to restrict these modular forms via a Siegel modular embedding
and translate by an appropriate theta characteristic.
We will determine this characteristic in Proposition~\ref{prop:WD6linequiv} resp.\  at the end 
of Section~\ref{sec:12pol} explicitly.

\paragraph{Good compactifications.} Let $\ol{X_D}$  (resp.\ $\ol{X_{D,(1,2)}}$) denote a good 
compactification of $X_D$  (resp.\ of $X_{D,(1,2)}$) in the sense of \cite{mumford77}
with boundary divisor $B$. Hirzebruch's minimal smooth compactification is
good, the one constructed by Bainbridge (\cite{bainbridge07}) to study the 
curves $W_D$ is good, too. Let $\omega_i$ for $i=1,2$ denote the line bundles 
of the natural foliations $\cF_1 = \{\HH \times \{pt\}\}$ and $\cF_2=\{\{pt\} \times \HH\}$ 
of the Hilbert modular surface $X_D$. They extend on the good compactification to line
bundles with the property (\cite{mumford77})
\be \label{eq:OMXD}
\Omega_{\ol{X_D}}^1(\log B) = \omega_1 \oplus \omega_2.
\ee
These compactifications will be used to perform intersection theory.
We denote the class of line bundles $\omega_i$ (or divisors like $W_D^X$) in the intersection ring 
of $\ol{X_D}$ (i.e.\ up to numerical equivalence) by $[\omega_i]$ (resp.\ by $[W_D^X]$).

\paragraph{\Teichmuller curves and Kobayashi curves.}
A {\em \Teichmuller curve} is an algebraic curve $C$ with a generically injective map $C \to \moduli$ 
to the moduli stack of curves that is a totally geodesic subvariety for the \Teichmuller
metric. On $\moduli$ the \Teichmuller metric agrees with
the Kobayashi metric and thus \Teichmuller curves are also Kobayashi curves in 
the following sense. 
\par
For any algebraic varity $Y$ we define a {\em Kobayashi curve $C$ in Y} to be an algebraic
curve $C$ together with a generically injective map $C \to Y$ that is totally geodesic
for the Kobayashi metric. In this paper we will apply this notion (besides for 
$\moduli$) only for Hilbert modular surfaces $X_D$. The universal covering of
$X_D$ is covered by Kobayashi geodesics, but only few of them descend to
algebraic curves, i.e.\ to Kobayashi curves on $X_D$.
\par
We recall the following characterization 
of Kobayashi curves from \cite{mvkobayashi}
to illustrate the various ways to interpret these curves. We will only need
the implication i) to iv) in the sequel. Note that we changed terminology from 
\cite{mvkobayashi}, where the notion Kobayashi geodesic was used for what
we now call Kobayashi curve.
\par
\begin{prop} Let $C \to X_{D,P}$ be an algebraic curve in a Hilbert modular surface for some
polarization $P$ with completion $\ol{C} \to \ol{X_{D,P}}$. Then the following conditions are
equivalent.
\begin{itemize}
\item[i)] The curve $C$ is a Kobayashi curve.
\item[ii)] The variation of Hodge structures (VHS) over $C$ has a rank two subsystem that
is maximal Higgs (see \cite{mvkobayashi} for the definitions).
\item[iii)] For (at least) one of the two foliation classes $\omega_i$ we have $[\omega_i][\ol{C}] = \chi(W_D^X)$.
\item[iv)] There exists (at least) one (of the two) natural foliations $\cF$ of $X_{D,P}$
such that the curve $C$ is everywhere transversal to $\cF$.
\item[v)] The inclusion $T_C(\log \ol{C} \setminus C) \to T_{X_{D,P}}(\log B)|_C$ splits.
\end{itemize}
\end{prop}
\par
\begin{proof}
The proofs are simplifications of \cite[Theorem~1.2]{mvkobayashi}, since for Hilbert modular surfaces 
the variation of Hodge structures already decomposes into rank two summands. Only these
two summands are candidates for the maximal Higgs sub-VHS.
\end{proof}
\par
\Teichmuller curves are generated as $\SL_2(\reals)$-orbits of {\em flat surfaces} $(X,\omega)$, 
i.e.\ of pairs of a Riemann surface $X$ and a non-zero holomorphic one-form $\omega$.
The moduli space of flat surfaces is a vector bundle minus the zero section over $\moduli$, 
denoted by $\omoduli$. Its quotient by the action of $\cx^*$ is denoted by $\pomoduli$. 
\par
The \Teichmuller curve is isomorphic to $\HH/\SL(X,\omega)$, where
$\SL(X,\omega)$ is the {\em affine group} of the flat surface, i.e.\ the group of matrix
parts of homeomorphisms of $X$ that are affine maps in the $\omega$ charts. At the
time of writing there is (still) no effective algorithm known to
determine $\SL(X,\omega)$, even if this group is known beforehand to be
a lattice in $\SL_2(\reals)$. Examples of flat surfaces generating \Teichmuller
curves are given in Figure~\ref{cap:WDLSX}. 
We explain more conceptually how these flat surfaces were constructed as Prym covers in the next section.
See e.g.\ \cite{masur02} or \cite{moelPCMI} for more background on 
flat surfaces, the $\SL_2(\reals)$-action and \Teichmuller curves.

\section{Prym varieties} \label{sec:prym}

A Prym curve is a (connected smooth algebraic) curve $X$ together with
an involution $\rho$ and a supplementary condition. This supplementary condition
is different in two of our primary references. In \cite{bl} it is
required that the Prym variety $\Prym(X) = \Prym(X,\rho)$ introduced
below naturally acquires a multiple of a principal polarization. In 
\cite{mcmullenprym} the author focusses on the case that $\Prym(X)$ is two-dimensional, i.e.\
an abelian surface. Although the first terminology seems to be more widely
used, we stick to the second since we are interested in Hilbert modular {\em surfaces}
and applications to curves on these surfaces.
\par
\paragraph{Prym varieties.}
Suppose that $X$ is a curve with an involution $\rho$ and we let  $Y = X/\langle \rho \rangle$.
The quotient map $\pi:X \to Y$ induces a map
$$q: \Jac(X) \to \Jac(Y)$$
and we let $\Prym(X)$ be the connected component of the identity of the kernel of $q$.
This abelian variety acquires a polarization $\cO_{\Jac(X)}(\Theta)|_{\Prym(X)}$
by restriction of the principal polarization on $\Jac(X)$. We call this
polarized abelian variety the {\em Prym variety} of $(X,\rho)$.
\par
Originally, Prym varieties were invented to construct more principally
polarized abelian varieties using curve (and covering) theory than just
Jacobians. Thus, it was required that the polarization on $\Prym(X)$
has a multiple of a principal polarization. This holds for \'etale double
covers and genus two double covers ramified at two points.

\paragraph{Double covers with two-dimensional Prym variety.}
In the sequel we use the modified terminology of \cite{mcmullenprym} and
say that $(X,\rho)$ is a Prym pair if $\dim(\Prym(X))=2$. This happens
for $g(X) = 3$ with $4$ ramification points, for $g(X)=4$ with
two ramification points and also for $g(X)=5$ with $\rho$ fixed point free.
The last case is less suitable for the construction of \Teichmuller curves, 
so we disregard it here.
\par
If $g(X)=4$, the principal polarization of $\Jac(X)$
restricts to a polarization of type $(2,2)$ on $\Prym(X)$ by \cite[Corollary 12.1.5]{bl}. That is, 
$\cO_{\Jac(X)}(\Theta)|_{\Prym(X)}$ is a positive line bundle on $\Prym(X)$
of type $(2,2)$. Consequently, there is a line  bundle $\cL$ on $\Prym(X)$
that defines a principal polarization and such that $\cL^{\otimes 2} = \cO_{\Jac(X)}(\Theta)|_{\Prym(X)}$.
\par
Since $\cL$ is a polarization, the Euler characteristic is the product of the $d_i$
appearing in the type and there is no higher cohomology, i.e.\ 
we have $\chi(\cL^{\otimes 2}) = 2\cdot 2 = H^0(\Prym(X),\cL^{\otimes 2})$. By Riemann-Roch, 
the self-intersection number of $\cL^{\otimes 2}$ equals $8$ and hence by
adjunction the vanishing locus of any section of $\cL^{\otimes 2}$ is a curve of
arithmetic genus five (hence possibly a curve of genus $4$ with 2 points
identified to form a node).
\par


\paragraph{The Prym \Teichmuller curves $W_D(6)$ and $W_D(4)$.} 
For $g(X) = 4$ the {\em Prym eigenform
locus $\Omega E_D$} is defined in \cite{mcmullenprym} to be the subvariety in $\omoduli[4]$ of
flat surfaces $(X,\omega)$ such that $X$ admits a Prym involution $\rho$ 
such that $\Prym(X)$ admits real multiplication by $\ord_D$ and, finally,  
such that $\omega$ is an eigenform for real multiplication by $\ord_D$. 
The eigenform condition includes in particular, that $\omega$ is in the $-1$-eigenspace
of the action of $\rho$ on $H^0(X,\Omega_X^1)$. The intersection of 
 $\Omega E_D$ with the minimal stratum $\omoduli[4](6)$ is shown in 
\cite{mcmullenprym} to be (complex) two-dimensional.  Its image in $\pomoduli[4]$ is an algebraic 
curve $W_D(6)$ that projects isomorphically to a \Teichmuller curve  $\moduli[4]$ that will also be denoted by $W_D(6)$. 
We do not assume here that $W_D(6)$ is irreducible, but by \cite{manhlann} this is conjecturally 
indeed the case. These curves can be generated by $X$-shaped flat surfaces 
(see Figure~\ref{cap:WDLSX}). We use this to label their images in Hilbert modular
surfaces by an upper index $X$, see below.
\par
Fully similarly, for $g(X)=3$  the {\em Prym eigenform
locus} $\Omega E_D$ is defined in \cite{mcmullenprym} to be the subvariety in $\omoduli[3]$ of
flat surfaces $(X,\omega)$ such that $X$ admits a Prym involution $\rho$ 
such that $\Prym(X)$ admits real multiplication by $\ord_D$ and, finally, 
such that $\omega$ is an eigenform for real multiplication by $\ord_D$. The
intersection of  $\Omega E_D$ with the minimal stratum $\omoduli[3](4)$ is shown in 
\cite{mcmullenprym} to be two-dimensional. Its image in $\pomoduli[3]$ is a curve $W_D(4)$,
that projects to a \Teichmuller curve in $\moduli[3]$ the will also be denoted by $W_D(4)$. 
These curves can be generated by $S$-shaped flat surfaces
(see Figure~\ref{cap:WDLSX}) and  we use this to label their images in Hilbert modular
surfaces by an upper index $S$, see below.
\par
\paragraph{The Abel-Jacobi image}
In the following lemma the non-hyperellipticity is still parallel for
$W_D(6)$ and $W_D(4)$, but in all further steps the two cases differ.
We thus concentrate on $W_D(6)$ for the rest of this section.
\par
If we fix a point $P$ on $X$, then the composition of the Abel-Jacobi map
$X \to \Jac(X)$ based at $P$ with the dual of the inclusion $\Prym(X) \to \Jac(X)$
defines a map $\varphi: X \to \Prym(X)$, called the {\em Abel-Prym map} (based at $P$).
Note for comparison with the $W_D^S$ case that in this (standard) definition we have used 
the principal polarizations of $\Prym(X)$ and $\Jac(X)$ to obtain a
map $\Prym(X) \to \Jac(X)$.
\par
\begin{lemma} \label{le:abelprymemb} 
The curves $W_D(6)$ and $W_D(4)$ are disjoint from 
the hyperelliptic locus. \par
Moreover, for each curve $[X] \in W_D(6)$ the Abel-Prym 
map is an immersion outside the fixed points of $\rho$ and maps the two fixed points 
of $\rho$ to a single point in $\Prym(X)$.
\end{lemma}
\par
\begin{proof}
Suppose $(X,\rho)$ was hyperelliptic with hyperelliptic involution $h$.
By the uniqueness of the hyperelliptic involution $\langle h,\rho \rangle \cong (\zed/2)^2$.
In particular $\tau = h\circ \rho$ is another involution. Moreover, by \cite{mumfordprym}
$\Prym(X) = \Jac(X/\tau)$ and hence the one-form $\omega$ is a pullback from $X/\tau$.
This contradicts that $\omega$ has a single zero of order $6$ resp.\ of order $4$.
\par
The disjointness from the hyperelliptic locus is the hypothesis needed
in \cite[Section~12.5]{bl} to deduce the remaining claims.
\end{proof}
\par
For the construction of modular forms using theta functions below we need a more precise
description of the Abel-Prym image in terms of the theta divisor on
$\Jac(X)$ and thus the theta divisor on $\Prym(X)$. Let $W^3 \subset \Pic^3(X)$ denote
the canonical theta-divisor in the Picard group of $X$, i.e.\ the image of the $3$-fold 
symmetric product of $X$
in $\Pic^3(X)$.
\par
\begin{prop} \label{prop:WD6linequiv}
Let $(X,\omega)$ be a surface in $W_D(6)$ and ${\rm div}(\omega)=6P$.
Then the spin structure determined by $\cO_X(3P)$ is even, 
i.e.\ $h^0(X,\cO_X(3P)) = 2$.
\par
Suppose the Abel-Prym map is based at $P$. Then the line bundles 
$\cO_A(\varphi(X))$ and $\cO_{\Jac(X)}(\Theta)|_{A} = \cL^2$ are
linearly equivalent, where $\Theta$ is the translate of $W^{3} \subset
\Pic^{3}(X)$ by $-3P$.
\end{prop}
\par
\begin{proof}
The possible configurations of cylinder decompositions of such a flat surface are
listed in the appendix of \cite{manhlann}. There are only two of them (one
is also visible in Figure~\ref{cap:WDLSX}). In both cases, the
parity of the spin structure can be calculated using the winding number
of a homology basis (as e.g.\ explained in \cite{kz03}) to be even. 
Clifford's theorem implies that for an even spin structure $h^0(X,\cO_X(3P)) = 2$.
\par
A non-hyperelliptic curve $X$ with an even spin structure has a unique divisor $D$
of degree three with $h^0(X,D) \geq 2$. This divisor is the unique singular
point of the theta divisor (considered in $\Pic^3(X)$).
\par
The algebraic equivalence of $\cO_A(\varphi(X))$ and $\cO_{\Jac(X)}(\Theta)|_{A} = \cL^2$
is a consequence of intersection theory, known as Welters' criterion (see \cite[Theorem~12.2.2]{bl}.
The translation image of a linear equivalence class on an abelian variety
is precisely the algebraic equivalence class. So we have to show that this
translation is zero for the two bundles in question. Equivalently, 
we have to show that $\varphi(X) \subset \Theta$, where we consider
$\varphi(X) \subset A \subset \Jac(X) = \Pic^{0}(X)$. Since $\varphi(X) \subset \Pic^0(X)$
is the set of classes $x-\rho(x)$, we have to show that $3P+x-\rho(x)$ is
an effective divisor for all $x \in X$. Since $h^0(X,\cO_X(3P)) = 2$, this is obviously true.
\end{proof}
\par
\paragraph{The curve $W_D^X$ in $X_D$.}
For any curve $X$ representing a point $[X] \in W_D(6)$ the Prym variety has
a principal polarization by $\cL$ defined above and real multiplication by $\ord_D$
by definition. We thus obtain a map $W_D(6) \to X_D$, whose image $W_D^X$ we now describe.
Recall that $du_1$ and $du_2$ are the eigenforms for real multiplication on the
abelian variety $A_{(z_1,z_2)}$ in the eigenform coordinates $\uu=(u_1,u_2)$ introduced in Section~\ref{sec:HMS}.
Since the abelian varieties $A_{(z_1,z_2)}$ and $A_{(z_2,z_1)}$ are isomorphic with
an isomorphism interchanging $du_1$ and $du_2$, we may assume in the sequel 
that $du_1$ is the eigenform with the $6$-fold zero. The other choice describes the flipped curve
$\tau(W_D^X)$, where $\tau(z_1,z_2) = (z_2,z_1)$.

\paragraph{Description of $W_D^X$ using theta functions.} 
By Lemma~\ref{le:abelprymemb} the Abel-Jacobi image of $X$ is cut out in $\Prym(X)$ by some
section in $\cL^{\otimes 2}$. If we fix a basis $\theta_0,\theta_1,\theta_2,\theta_3$
of these sections, it is cut out  by the vanishing of
$$\theta_X(\zz,\uu) = \sum_{j=0}^3 a_j(\zz) \theta_j(\zz,\uu)$$
for some choice of coefficients $a_j(\zz)$. We now determine the coefficients $(a_0:\ldots:a_3)$,
as a projective tuple such that $\varphi(X) = \{\theta_X = 0\}$.
\par
First, by the choice of the base point of the Abel-Prym map, the divisor $\theta_X$ has 
to contain zero, hence
\be \label{eq:throughzero}
\sum_{j=0}^3 a_j(\zz) \theta_j(\zz,0) =0
\ee
From Lemma~\ref{le:abelprymemb} we deduce that $\varphi(X)$ has an ordinary double point at 
$\uu=(0,0)$. 
\par
\begin{lemma} \label{le:6foldzero}
Fix the point $\zz \in \HH^2$ and assume that the $a_{j}(\zz)$ are chosen 
such that \eqref{eq:throughzero} holds. Then the 
differential $du_1$ restricted to $X$ has a zero of order $6$ at $\uu=(0,0)$ if and only if
$\frac{\partial^k}{\partial u_2^k}\theta_X(\zz,\uu)|_{\uu=(0,0)}$ vanishes for $k=1,\ldots,6$.
\end{lemma}
\par
\begin{proof}
If $\varphi(X)$ had a single branch through the origin $(0,0)$ in $A_\zz$ the vanishing of
these partial derivatives would be exactly a reformulation of an eigenform having
a $6$-fold zero at that point $(0,0)$. But since by Lemma~\ref{le:abelprymemb} there are two
branches through $(0,0)$  the vanishing of these partial derivatives only implies
that ${\rm ord}_{(0,0)}(du_1) \geq 5$. But since $(0,0)$ is a fixed point of $\rho$
and $du_1$ in the $(-1)$-eigenspace of $\rho$ the vanishing order ${\rm ord}_{(0,0)}(du_1)$
is even, thus proving the claim. 
\end{proof}
\par
We use the shorthand notation $\theta_j(\zz) = \theta_j(\zz,(0,0))$ and the classical
terminology {\em theta constants} for these
restrictions. Indices of theta constants are to be read modulo $4$ in the sequel. For a function
$f(\zz,\uu)$ the prime denotes the derivative 
$f'(\zz)= \frac{\partial}{\partial z_2}f(\zz,\uu)|_{\uu=(0,0)}.$
\par
\begin{prop} \label{prop:GD}
The function
\be  \label{eq:GD} G^X_D(\zz) = 
\sum_{j=0}^3 a_j(\zz) \theta_j'''(\zz), \,
\text{where} \,\, a_j(\zz) = \left| \begin{matrix} 
\theta_{j+1}(\zz) & \theta_{j+2}(\zz) & \theta_{j+3}(\zz) \\
\theta'_{j+1}(\zz) & \theta'_{j+2}(\zz) & \theta'_{j+3}(\zz) \\
\theta''_{j+1}(\zz) & \theta''_{j+2}(\zz) & \theta''_{j+3}(\zz) \\
\end{matrix} \right|
\ee
is a modular form of weight $(2,14)$ for $\SLOD$ and for some 
character. Its vanishing locus contains $W_D^X$.
\end{prop}
\par
Equivalently, we can write 
\be  \label{eq:GDnew} G^X_D(\zz) = 
 \left| \begin{matrix} 
\theta_{0}(\zz) & \theta_{1}(\zz) & \theta_{2}(\zz) & \theta_{3}(\zz)\\
\theta'_{0}(\zz) & \theta'_{1}(\zz) & \theta'_{2}(\zz) & \theta'_{3}(\zz)\\
\theta''_{0}(\zz) & \theta''_{1}(\zz) & \theta''_{2}(\zz) & \theta''_{3}(\zz) \\
\theta'''_{0}(\zz) & \theta'''_{1}(\zz) & \theta'''_{2}(\zz) & \theta'''_{3}(\zz) \\
\end{matrix} \right|
\ee
Hence (up to a scalar) $G^X_D(\zz)$ does not depend on the choice of the basis of $\cL^{\otimes 2}$.
\par
\begin{proof}
The functions $\theta_j(\zz)$ are modular forms of weight $(1/2,1/2)$ for
a subgroup of $\Gamma$ finite index of $\SLOD$. (For Siegel theta functions
of second order this group is the congruence group $\Gamma(4,8)$, so
$\Gamma \supset \psi^{-1}(\psi(\SLOD) \cap \Gamma(4,8))$, where $\psi$ is
a Siegel modular embedding. The precise form of $\Gamma$ will play no role.)
First, we claim that the $a_j(\zz)$ are modular forms of weight $(3/2,15/2)$ for
the same subgroup $\Gamma$. This is a general principle, extending Rankin-Cohen
brackets (see e.g.\ \cite{zagier123}) to three-by-three determinants. 
Roughly, the derivative of a modular form $f$
of weight $(k_1,k_2)$ in the second variable is a modular form of weight
$(k_1,k_2+2)$ plus a non-modular contribution.  These non-modular contributions
cancel when taking the linear combinations that appear in a determinant. All
the summands in the determinant are of weight $(1/2,1/2)+(1/2,5/2)+(1/2,9/2)$, 
thus proving the claim. More precisely, for any $\mat abcd \in \SLOD$ we have
\be \label{eq:derivMF}
\begin{aligned}
\frac{\partial}{\partial z_2}f(\frac{az_1 + b}{cz_1 + d},
\frac{a^\sigma z_2 + b^\sigma}{c^\sigma z_2 + d^\sigma })  
& = (cz_1+d)^{k_1}(c^\sigma z_2 +d^\sigma)^{k_2+2} 
\frac{\partial}{\partial z_2}f(z_1,z_2) \\  &+ k_2 c^\sigma (cz_1+d)^{k_1}(c^\sigma z_2 +d^\sigma)^{k_2+1} f(z_1,z_2).
\end{aligned}
\ee
Hence $\left(\frac{\partial f}{\partial z_2}\right)g$ differs from a modular
form in a summand  $$k_2 c^\sigma (cz_1+d)^{k_1}(c^\sigma z_2 +d^\sigma)^{k_2+1} f\cdot g.$$
Since this summand is symmetric in $f$ and $g$, if $f$ and $g$ have the same
weight and thus the same $k_2$, we deduce the modularity of $fg'-f'g.$ Differentiating
\eqref{eq:derivMF} once more gives the result for second derivatives and three-by-three determinants.
\par
By Lemma~\ref{le:6foldzero} on the locus $W_D^X$ the section
$\frac{\partial^k}{\partial u_2^k}\theta_X(\zz,\uu)|_{\uu=(0,0)}$ vanishes for $k=1,\ldots,5$.
Vanishing for odd $k$ is automatic, since all the $\theta_i$ and hence
$\theta_X$ is even. By the heat equation 
$\frac{\partial^{2k}}{\partial u_2^{2k}}\theta_i(\zz,\uu)|_{\uu=(0,0)} =
\frac{\partial^{k}}{\partial z_2^{k}}\theta_i(\zz)$. Thus the $a_j(\zz)$
have to satisfy 
$$\frac{\partial^{2k}}{\partial u_2^{2k}}\left(\sum_{j=0}^3 a_j(\zz) \theta_j(\zz,\uu)|_{\uu=(0,0)}\right)= 
\sum_{j=0}^3 a_j(\zz) \frac{\partial^{k}}{\partial z_2^{k}}\theta_j(\zz) = 0$$
for $k=0,1,2$. The $a_j(\zz)$ given in \eqref{eq:GD} are, up to a common scalar factor, 
the unique solution to these conditions. 
\par
We now use the same argument derived from differentiating \eqref{eq:derivMF} again.
Since $\theta_i(\zz,\uu)$ are modular forms of weight $(1/2,1/2)$ and
the $a_j(\zz)$ are modular forms of weight $(1/2,1/2)$, the sum
$G_D^X$ is a modular form of weight  $(3/2,15/2)+(1/2,13/2)=(2,14)$ 
plus a multiple of $\sum a_j(\zz) \theta_j''(\zz)$, which is known to vanish
by the choice of the $a_j(\zz)$.
\par
Finally, we consider the action of $\SLOD/\Gamma$. It is known for Siegel theta
functions that $\Gamma(2)/\Gamma(4,8)$ acts by characters and that 
$\Sp(2g,\zed)/\Gamma(2)$ acts by a linear representation on the basis of
$\cL^{\otimes 2}$. These statements obviously also hold for $\psi(\SLOD)$.
From the determinantal form of $G_D^X$ it is obvious that the change of
basis leaves $G_D$ unchanged. Consequently, $G_D^X$ is a modular form
for the full Hilbert modular group $\SLOD$ for some character. 
\end{proof}
\par
\paragraph{Explicit construction.} We may take the classical theta functions
\eqref{eq:thclass} and restrict them via a Siegel modular embedding.
By Proposition~\ref{prop:WD6linequiv} no translation by a characteristic is required.
Consequently, by \eqref{eq:thbasis} applied to $P = \mat02{-2}0$ and $c_1=c_2=0$ 
we obtain the desired basis $\theta_0,\ldots,\theta_3$ used to construct $G_D^X$.

\section{A Prym variety with $(1,2)$-polarization.} \label{sec:12pol}

The aim of this section is a characterization of $W_D^S$ in terms
of derivatives of theta functions, parallel to Proposition~\ref{prop:GD}.
The corresponding modular form is constructed in Proposition~\ref{prop:GDS}.
Consequently, in this section we restrict to  $g(X)=3$ and $\rho$ is an involution
on $X$ with $4$ fixed points.
\par
{\bf The Abel-Jacobi map revisited.} Let 
$\iota:\Prym(X) \to \Jac(X)$ be the inclusion of the Prym
variety, defined in the preceding section as the connected
component of the identity of the kernel of  $q: \Jac(X) \to \Jac(Y)$.
The restriction of the principal polarization of $\Jac(X)$
to $\Prym(X)$ is a polarization of type $(1,2)$ that we
denote by $\cL = \cO_{\Jac(X)}(\Theta)|_{\Prym(X)}$. Now, 
the canonical map $\phi_\cL: \Prym(X) \to \Prym^\vee(X)$
associated with $\cL$ is no longer an isomorphism, but
of degree $4$. Consequently, there is a dual isogeny
$(\phi_\cL)^\vee:  \Prym(X) \to \Prym^\vee(X)$ with the property
that $(\phi_\cL)^\vee \circ \phi_\cL = [2]$ is the multiplication by two map.
The map $(\phi_\cL)^\vee$ is induced by a line bundle $\check{\cL}$
on $\Prym^\vee(X)$ which is also a polarization of type $(1,2)$ (see
\cite{bipol}). The map $(\phi_\cL)^\vee = \phi_{\check{\cL}}$ depends only on the image
of $\check{\cL}$ in the N\'eron-Severi group, i.e.\ for the moment $\check{\cL}$
is well-defined only up to translations.
\par
Still identifying $\Jac(X)$ with its dual we have the dual inclusion map
$\check{\iota}: \Jac(X) \to \Prym^\vee(X)$ and the Abel-Jacobi map
$\varphi$ defined above generalizes as the composition of the map $X \to \Jac(X)$ (still
depending on the choice of a base point) of with $\phi_{\check{\cL}} \circ
\check{\iota}$. We let $\varphi_0$ be the composition of  $X \to \Jac(X)$
and $\check{\iota}$, so that $\varphi=\phi_{\check{\cL}} \circ \varphi_0$. 
We call $\varphi_0$ the {\em pre-Abel-Jacobi map} and $\varphi_0(X)$
the {\em pre-Abel-Jacobi image} of $X$. Most of the following lemma is
also covered by results in \cite{barth12pol}. 
\par
\begin{lemma} \label{le:preAJAJ}
The pre-Abel-Jacobi image of $X$ is embedded into $\Prym^\vee(X)$. The image of the fixed
points are two-torsion points in $\Prym^\vee(X)$.
\par
The Abel-Jacobi image of $X$ is embedded into $\Prym(X)$ outside the four
fixed points of $\rho$. These four fixed points are mapped to $0 \in \Prym(X)$.
\par
The map $(-1)$ on $\Prym^\vee(X)$ induces the involution $\rho$ on $X$.
\end{lemma}
\par
\begin{proof}
Recall that the map $q^\vee: \Jac(Y) \to \Jac(X)$ is given for any degree
zero divisor $D$ by $D \mapsto D + \rho(D)$.
The map $\check{\iota}$ is the quotient map by the image of $q^\vee$.
Hence two points $D_1$ and $D_2$ are the same in $\Prym^\vee(X)$, if and only if they
differ by an element of the form  $D + \rho(D)$. This implies that
a fixed point of $\rho$ maps to a point of order two in $\Prym^\vee(X)$ and the last statement.
\par
The Prym variety $\Prym(X)$ is the complementary subvariety to $\Jac(Y)$
inside $\Jac(X)$ in the sense of \cite[Section~12.1]{bl}. Consequently, 
$\Prym(X)$ is the image of $(1+\rho): \Jac(X) \to \Jac(X)$. Suppose that
$\varphi(x)=\varphi(y)$. Then $(x-P)+(\rho(x)-P) \sim (y-P)+(\rho(y)-P)$,
hence $x-\rho(x) \sim y - \rho(y)$. Since $X$ is not hyperelliptic, 
this can only happen if $x$ and $y$ are both fixed point of $\rho$.
This implies the pointwise injectivity for both $\varphi_0$ and $\varphi$
outside the fixed points.
\par
Let $P_i$ denote the images in $Y$ of the fixed points of $\rho$.
The projectivised differential of the Abel-Prym map is the composition
$X \to Y \to \proj(H^0(\cO_Y \otimes \eta))$ where $\eta$ is a
line bundle defining the double covering, i.e.\ 
$\eta^{\otimes 2} = \cO_Y(P_1+P_2+P_3+P_4)$. The Abel-Prym map
is not an embedding at a point $x$, if and only if $x$ is a base point of
$\cO_Y \otimes \eta$, i.e.\ if $h^0(\cO_Y \otimes \eta) = h^0(\cO_Y \otimes \eta (-x))$.
(Details on both statements can be found in \cite[Proposition~12.5.3 and Corollary~12.5.5]{bl}, 
the principal polarization hypothesis is not used.) On a curve of genus one the bundle $\cO_Y$
is trivial and since $\deg(\eta) = 2$, Riemann-Roch implies the claim.
\par
It remains to show that the images of the fixed points of $\rho$  are
actually distinct in $\Prym^\vee(X)$. If not, then for some fixed point $Q$ we have
$Q-P \sim D + \rho(D)$ for some degree zero divisor $D$ on $X$. Since $Y$ is an elliptic curve, hence
equal to its Jacobian, we may moreover suppose that $D \sim (R-P)$ for some point $R \in X$. Together
we obtain $Q+P \sim R + \rho(R)$, which contradicts that $X$ is not hyperelliptic.
\end{proof}
\par
\begin{lemma} For a flat surface $(X,\omega)$ parameterized by
$W_D(4)$ the parity of the spin structure is odd.
\end{lemma}
\par
\begin{proof}
This can be checked on any flat surface representing $(X,\omega)$
using the winding number definition given in \cite{kz03}. Alternatively,
we can use the classification of strata in \cite{kz03} together
with Lemma~\ref{le:abelprymemb} stating that we are not in the hyperelliptic
stratum.
\end{proof}
\par
We denote by $K(\check{\cL}) \subset \Prym^\vee(X)$ the kernel of $\phi_{\check{\cL}}$. 
It is shown in  \cite[Lemma~10.1.2]{bl} that a polarization $\check{\cL}$ of type $(1,2)$ 
on an abelian variety has exactly $4$ base points of the linear system $|\check{\cL}|$, i.e.\ there are
exactly $4$ points that all the vanishing loci of sections in $H^0(\check{\cL})$ have in common. Moreover, 
the $4$ base points form one orbit under the translation by  $K(\check{\cL})$.
\par
\begin{lemma} \label{le:cLclass}
There is a unique choice of $\check{\cL}$ within its algebraic equivalence class
such that zero is a base point of $\check{\cL}$.
\par
With this choice of $\check{\cL}$, if the flat surface $(X,\omega)$ lies in $W_D(4)$, 
then the pre-Abel-Prym image $\varphi_0(X)$ is the vanishing locus
of some section in $H^0(\check{\cL}))$. Moreover  
$(-1)^*\check{\cL} = \check{\cL}$ with this choice of $\check{\cL}$, 
and all the global sections of $\check{\cL}$ are odd.
\end{lemma}
\par
A similar statement does not seem to hold for the Abel-Prym
image. By Lemma~\ref{le:preAJAJ} it has arithmetic genus $6$ and
the obvious polarizations on $\Prym(X)$ are of type $(1,2)$ 
or maybe $(1,4)$. By adjunction the vanishing loci of
their global sections are curves of arithmetic genus $3$ or $9$
respectively.
\par
In the sequel we use the endomorphism $\delta(C,D)$ associated
with a curve $C$ and a divisor $D$ of an abelian variety $A$. It is
defined by mapping $a \in A$ to the sum of intersection points of the curve
$C$ translated by $a$ and the divisor $D$.
\par
\begin{proof}
Recall that the algebraic equivalence class of a line
bundles consists exactly of its translations by any point in $\Prym(X)^\vee$.
Two translations of $\check{\cL}$ that have zero as base point differ by
an element in $K(\check{\cL})$. But for $c \in K(\check{\cL})$ we have
$t_c^* \check{\cL} \sim \check{\cL}$ by definition of $\phi_{\check{\cL}}$.
This proves the first statement.
\par
For the second statement  we first show that $\varphi_0(X)$ and $\check{\cL}$  are algebraically
equivalent. Following the strategy in \cite[Lemma 12.2.3]{bl} 
we need to show by \cite[Theorem~11.6.4]{bl} that 
$\delta(\varphi_0(X), \check{\cL}) = \delta(\check{\cL},\check{\cL})$. By \cite[Proposition~5.4.7]{bl}
we have $\delta(\check{\cL},\check{\cL}) = -2 {\rm id}_{\Prym^\vee(X)}$. On the other
hand, still identifying $\Jac(X)$ with its dual using the polarizations $\Theta$ we
have 
$$\delta(\varphi_0(X),\check{\cL}) = - \check{\iota} \circ \iota \circ \phi_{\check{\cL}} =
  - \phi_{\cL} \circ \phi_{\check{\cL}}  = -2 {\rm id}_{\Prym^\vee(X)}
$$
by \cite[Proposition~11.6.1]{bl}, by definition of $\cL$ as restriction of $\Theta_{\Jac(X)}$
and by definition of $\check{\cL}$.
\par
Next, we check that we correctly normalized $\check{\cL}$ within its algebraic equivalence class.
If two points in $\varphi_0(X)$ differ by an element in $K(\check{\cL})$, they are mapped
to the same point in $\varphi(X)$. By Lemma~\ref{le:preAJAJ} this happens for any two points among 
the fixed points of $\rho$ and for no other pair of points. Since the base points of
$\check{\cL}$ differ by elements in $K(\check{\cL})$, this implies the claim.
\par
The last statement is a special case of the results in \cite[Sections~4.6 and~4.7]{bl}.
\end{proof}
\par
\paragraph{The curve $W_D^S$ in $X_{D,(1,2)}$.}
We keep the normalization of $\check{\cL}$ within its algebraic equivalence class from now on.
For any curve representing a point $[X] \in W_D(4)$ the Prym variety has
a  polarization by $\cL$ of type $(1,2)$ defined above and real multiplication by $\ord_D$
by definition. We thus get a map $W_D(4) \to X_{D,(1,2)}$, whose image $W_D^S$ we now describe.
Recall that $du_1$ and $du_2$ are the eigenforms for real multiplication on the
abelian variety $A_{(z_1,z_2)}$ in the eigenform coordinates $\uu$ introduced in Section~\ref{sec:HMS}.
\par
For $D \equiv 0\, (4)$ we may suppose that $du_1$ is the eigenform and that
the abelian varieties in  $X_{D,(1,2)}$ are $A_\zz = \cx^2/\Lambda^\fra_\zz$, 
compare \eqref{eq:isotau}. For $D \equiv 1 \, (4)$, however, we have to replace $\fra$
by $\fra^\sigma$, when insisting on this normalization. This will play a role in the next section. 
\par
We can now strengthen the previous lemma using the real multiplication and the 
$4$-fold zero condition.
\par
\begin{lemma} \label{la:partder}
If the image of the flat surface $(X,\omega)$ lies in $W_D^S$, then 
there exists some section $\theta_X$ of $H^0(A_\zz,\check{\cL})$
such that the partial derivatives
$$\frac{\partial^k}{\partial u_2^k} \theta_X(\zz,\uu) $$
for $k=0,\ldots,4$  vanish at the point $\uu = (0,0)$.
\end{lemma}
\par
Since $c$ was chosen in the base locus of $\check{\cL}$ and since the sections
of $\check{\cL}$ are odd, the vanishing of the derivatives
for $k=0,2,4$ is automatic. We now express the vanishing for $k=1$ and $k=3$ in terms
of theta functions.
\par
We now fix a basis $\{\theta_0(\zz,\uu),\theta_1(\zz,\uu)\}$ of sections of $\check{\cL}$.
We define 
$$D_2\theta_j(\zz) = \frac{\partial}{\partial u_2} \theta_j(\zz,\uu)|_{\uu=(0,0)}
\quad \text{for} \quad j=0,1$$ and for analogous purposes as in 
Section~\ref{sec:prym} we let $$a_j(\zz) = (-1)^j D_2 \theta_{j+1}(\zz).$$ Recall from the
previous section the definition $f'(\zz)= \frac{\partial}{\partial z_2}f(\zz,\uu)|_{\uu=(0,0)}.$
\par
\begin{prop} \label{prop:GDS} Let $\fra$ be a fractional $\ord_D$-ideal of norm two.
Then the function
\be \label{eq:GDS}
\begin{aligned}
G_{D}^S &= a_0(\zz)\frac{\partial^3}{\partial u_2^3} \theta_0(\zz,\uu)|_{\uu=(0,0)} + 
a_1(\zz)\frac{\partial^3}{\partial u_2^3} \theta_1(\zz,\uu)|_{\uu=(0,0)} \\
& = 
\left|\begin{matrix} D_2\theta_0(\zz) & D_2 \theta_1(\zz) \\ D_2\theta'_0(\zz) & D_2 \theta'_1(\zz) \\
\end{matrix}\right|
\end{aligned}
\ee
is a modular form of weight $(1,5)$ for the Hilbert modular group $\SL(\fra \oplus \ord^\vee)$.
The union of the vanishing loci of $G_{D}^S$ for the one or two choices of $\fra$ (depending
on $D \equiv 0\,\,(4)$ or $D \equiv 1\,\,(8)$)  
contains the image $W_D^S$ of the \Teichmuller curve $W_D(4)$ in the real multiplication locus $X_{D,(1,2)}$.
\end{prop}
\par
\begin{proof}
By the definition of $c$ both $\theta_0$ and $\theta_1$ vanish at
$\uu = (0,0)$. Since the $\theta_j$ are modular forms of weight $(1/2,1/2)$
it is an immediate consequence of the transformation formula for theta
functions and this vanishing property that $D_2\theta_j$ is a 
modular forms of weight $(1/2,3/2)$ for some subgroup of the Hilbert modular
group. By the principle of the construction of Rankin-Cohen brackets, 
$G_D^S$ is a modular form of weight $(1/2,3/2)+(1/2,7/2)=(1,5)$, as claimed.
\par
For the second statement, note that  $\theta_X (\zz,\uu)= a_0(\zz) \theta_0(\zz,\uu) + 
a_1(\zz)\theta_1(\zz,\uu)$ is a section of $\check{\cL}$ whose first two
partial derivatives in the $u_2$-direction vanish at $(0,0)$. The vanishing
of $G_{D}^S$ implies the vanishing of the third partial derivative and also the
forth derivative vanishes, since $\theta_X$ is odd. Lemma~\ref{la:partder}
now implies the claim.
\end{proof}
\par
\paragraph{Explicit construction.} The classical theta functions \eqref{eq:thclass} 
with characteristic $c=0$ are even
functions in $\vv$ (and $\uu$). More generally, a shift by a half-integral characteristic 
$c = (c_1,c_2) \in (\frac12 \zed^2)^2$ is an even function, if $4c_1^T c_2$ is even
and odd otherwise. Consequently, for a $(1,2)$-polarized abelian variety, both
$\theta_0 = \theta_{(\frac12,0),(\frac12,0)}$ and $\theta_1 = \theta_{(\frac12,\frac12),(\frac12,0)}$
are odd functions, and sections of the same line bundle by \eqref{eq:thbasis}. Since these
functions are odd, zero is a base point of the global sections of the 
theta line bundle shifted by the characteristic $c = Z \left(\begin{matrix}
1/2 \\ 0 \end{matrix}\right) + \left(\begin{matrix}
1/2 \\ 0 \end{matrix}\right)$. Consequently, 
$\theta_0$ and $\theta_1$ are an explicit form of the basis needed to construct $G_D^S$.

\section{Euler characteristic of $W_D^X$ and $W_D^S$}
\label{sec:eulerchar}

The preceding theta-function interpretation gives a way to 
calculate the Euler characteristic of $W_D^X$ and $W_D^S$, and of
the \Teichmuller curves $W_D(6)$ and $W_D(4)$.  In this section we prove
all the results announced in the introduction and cross-check our calculations
of Euler characteristics with the examples in \cite{mcmullenprym}. We invite the
reader to compare this result to the work of Bainbridge \cite{bainbridge07}. He
computed the Euler characteristics of the \Teichmuller curves $W_D$ using a modular
form that he defined using flat geometry. His formula is
\be
\chi(W_D) = - \frac92\chi(X_D).
\ee
Subsequently, a theta function construction
of his modular form was found in \cite{moelzag}.
\par
\begin{theorem} \label{thm:eulerX}
For the Kobayashi curves $W_D^X$ and for the Prym \Teichmuller curves $W_D(6) \subset \moduli[4]$
the Euler characteristic is given by
$$ \chi(W_D^X) = \chi(W_D(6)) = -7 \chi(X_D).$$
\end{theorem}
\par
The explicit formula for $\chi(X_D)$ is stated in \eqref{eq:chiXD}.
\par
\begin{example} {\rm For $D=8$ we have $\chi(X_8) = -1/6$ and thus
$\chi(W_8(6))= 7/6 $. 
This corresponds to the curve $V(X_1)$ calculated in \cite{mcmullenprym} with genus zero, two
cusps and two elliptic points of order two and three respectively. 
\par
For $D=12$ we obtain $\chi(X_{12}) = -1/3$ and thus $\chi(W_{12}(6))=  7/3$.  This corresponds
to the curve $V(X_3)$ calculated in \cite{mcmullenprym} with genus zero, three
cusps, one elliptic point of order two and one elliptic point of order six.}
\end{example}
\par
\begin{theorem} \label{thm:eulerS}
For genus three and $D \equiv 5 \, (8)$ the locus $W_D(4)$ is empty
(and $W_D^S$ is not defined). For $D \equiv 4 \, (8)$ we have 
$$ \chi(W_D^S)  = \chi(W_D(4)) = -\frac52 \chi(X_{D,(1,2)}).$$
For $D \equiv 1 \, (8)$ there are two components $W_D(4)^1$ and 
$W_D(4)^2$, each mapping to a Kobayashi curve $W_D^{S,j}$ in $X_{D,(1,2)}$
and the  Euler characteristic is given by
$$\chi(W_D^{S,j}) = \chi(W_D(4)^j) = -\frac52 \chi(X_{D,(1,2)}) \quad \text{for}\quad j=1,2.$$
\end{theorem}
\par
\begin{example} {\rm For $D=12$ we have $\chi(X_{12}) = \chi(X_{12,(1,2)})$
by Proposition~\ref{prop:vol12}.
Hence $\chi(X_{12,(1,2)}) = -1/3$ as in the preceding example and thus
$\chi(W_{12}^S) = 5/6$. By \cite{manhlann} the locus $W_{12}(4)$ for this discriminant
has one component only. It corresponds
to the curve $V(S_1)$ calculated in \cite{mcmullenprym} with genus zero, two
cusps and one elliptic points of order six. 
\par
For $D=20$ we have  
$\chi(X_{20,(1,2)}) = 3/2 \chi(X_{20})$ by Proposition~\ref{prop:vol12}.
Consequently, $\chi(X_{20,(1,2)}) = 1$ and thus $\chi(W_{20}^S) =  5/2$.  
By \cite{manhlann} the locus $W_{20}(4)$ for this 
discriminant has one component only. It corresponds
to the curve $V(S_2)$ calculated in \cite{mcmullenprym} with genus zero, four
cusps and one elliptic point of order two.}
\end{example}
\par
The proof of both theorems will be completed at the end of this section.
\par
\paragraph{A Torelli-type theorem.} The Prym-Torelli map
associates with any Prym pair $(X,\rho)$ (or equivalently with
a (quotient) curve $Y$ and the covering datum) its Prym variety.
For $g(X) = 4$ we have $g(Y)=2$ with two fixed points, thus a
moduli space of dimension $5$. Since $\dim \AVmoduli[2] = 3$, there
cannot exist a Torelli theorem retrieving the curve from 
its Prym variety. Nevertheless, a corresponding statement
holds when restricted to real multiplication and eigenforms
with a zero of high multiplicity.
\par
\begin{prop} \label{prop:g4torelli}
Let $(A_\zz,\cL)$ be a principally polarized abelian surface with real multiplication
by $\ord_D$ and suppose that the corresponding point $[A_\zz] \in X_D$ is 
in the vanishing locus of $G_D^X$. Then there
is one and only one curve $X$ of genus four with $[X] \in W_D(6)$ whose Prym variety is $A_\zz$
and such that the eigenform with a $6$-fold zero is $du_1$.
\par
In particular the vanishing locus of $G_D^X$ is equal to $W_D^X$.
\end{prop}
\par
\begin{proof}
We reverse the reasoning in the construction of $G_D^X$. Take a basis
$\theta_0(\zz,\uu),\ldots,\theta_3(\zz,\uu)$ of sections of $\cL^{\otimes 2}$ and choose
$a_j(\zz)$ as in \eqref{eq:GD} (where we now work with partial derivatives 
in the $u_2$-direction of order $0$,$2$ and $4$, to avoid $z_2$-derivatives
in this pointwise argument). Now let
$$\theta_X(\zz,\uu) = \sum_{j=0}^3 a_j(\zz) \theta_j(\zz,\uu)$$
and define $X'_\zz$ to be the vanishing locus of $\theta_X$. One easily checks
that the function $\theta_X$
does not depend on the choice of the basis for $\cL^{\otimes 2}$.
\par
Suppose that $\theta_X$ is not zero. By Riemann-Roch and adjunction, $X'_\zz$ is a curve of
arithmetic genus $5$. Since all these sections  of  $\cL^{\otimes 2}$ 
are even and since $X'_\zz$ passes through the origin of $A_\zz$, 
it has a singularity there. The vanishing of the derivatives implies
that on the branch of $X'_\zz$ at zero in the direction $u_1$ the
one-form $du_1$ has a zero of order at least $5$. Hence the geometric
genus of $X'$ is at least four and hence the singularity at
the origin is just a normal crossing of two branches. Since $\theta_X$
is even, $du_1$ has in fact a zero of order $6$ at the origin.
By Welters' criterion for a curve to generate a Prym variety (\cite[Theorem~12.2.2]{bl})
the normalization $X_\zz$ of $X'_\zz$ is a curve of genus four with
an involution $\rho$ induced by $\uu \mapsto -\uu$ and $A_\zz$ is
the Prym variety of $(X,\rho)$. We conclude that $X_\zz \in W_D(6)$. This shows  
that there is a curve $W_D(6)$ whose Prym image is $A_\zz$ and by the
argument leading to Proposition~\ref{prop:GD} the curve $X$ just constructed is
the only choice with $du_1$ restricting to a zero of order $6$, if we can rule out that
all the $a_j(\zz)$ are zero, which is equivalent to the assumption $\theta_X \neq 0$.
Together with the inclusion stated in Proposition~\ref{prop:GD} this concludes the
proof of the proposition under the assumption on $\theta_X$.
\par
Suppose that $\theta_X$ was zero for some $\zz \in \HH^2$. This implies that
all the $a_j(\zz)$ vanish. Consequently, if we let
\be M =   \left( \begin{matrix} 
\theta_{0}(\zz) & \theta_{1}(\zz) & \theta_{2}(\zz) & \theta_{3}(\zz)\\
\theta'_{0}(\zz) & \theta'_{1}(\zz) & \theta'_{2}(\zz) & \theta'_{3}(\zz)\\
\theta''_{0}(\zz) & \theta''_{1}(\zz) & \theta''_{2}(\zz) & \theta''_{3}(\zz) \\
\end{matrix} \right), \quad \text{then} \quad {\rm rank}(M) \leq 2.
\ee
The image of $A_\zz$ under the projective embedding $A_\zz \to \proj^3$ defined by
the sections of $\cL^{\otimes 2}$
is known to be a Kummer surface, the quotient of $A_\zz$ by the involution
$(-1)$ (see \cite[Section~10]{bl}). Such a Kummer surface has $16$ nodes 
at the images of two-torsion points, i.e.\ singular points with local 
equation $x^2+y^2+z^2=0$. Since the Hessian of such a singularity has non-zero
determinant, this contradicts the above hypothesis ${\rm rank}(M) \leq 2$.
\end{proof}
\par
The same line of arguments works for $g=3$, with a different geometric
argument to rule out $\theta_X = 0$ and with an extra 
twist due to the decomposition behavior of the prime two.
\par
\begin{prop} \label{prop:g3torelli} Fix a fractional $\ord_D$-ideal $\fra$ of norm two and 
a realization $ X_{D,(1,2)} \cong \HH^2/\SL(\fra \oplus \ord_D^\vee)$.
Let $(A_\zz,\cL)$ be a $(1,2)$-polarized abelian surface with real multiplication
by $\ord_D$ and suppose that the corresponding point $[A_\zz] \in X_{D,(1,2)}$ is 
in the vanishing locus of $G_D^S$. Then there
is one and only one curve $X$ of genus four with $[X] \in W_D^X$ whose Prym variety 
is $[A_\zz]$ and such that the eigenform with a $4$-fold zero is $du_1$.
\par 
Moreover, if $D \equiv 1\, (8)$, the preimages $W_D(4)^1$ and $W_D(4)^2$ in $\moduli[4]$
of the vanishing locus of $G_D^S$ for the two choices $\fra$ and $\fra^\sigma$ of a prime ideal of norm two
are generically different.
\end{prop}
\par
\begin{proof} To prove the first statement, we reverse the argument of Proposition~\ref{prop:GDS}
and use the notations introduced there. 
Fix a basis $\theta_0(\zz,\uu), \theta_1(\zz,\uu)$ of sections of $\cL$ 
and consider $\theta_X = a_0(\zz) \theta_0 + a_1(\zz) \theta_0$.
Suppose that $\theta_X$ is not zero. Then the vanishing locus $X = \{\theta_X = 0\}$ is a curve
of arithmetic genus three and by construction $du_1$ is a holomorphic one-form on $X$
with a zero of order (at least) $4$ at $0 \in X$. This implies that $X$ is smooth. Since
by Lemma~\ref{le:cLclass} the map $(-1)$ on $A_\zz$ induces an involution on $X$ with $4$
fixed points, we conclude that $[X] \in W_D(4)$. This shows  
that there is a curve $W_D(4)$ whose Prym image is $A_\zz$ and by the
argument leading to Proposition~\ref{prop:GDS} the curve $X$ just constructed is
the only choice with $du_1$ restricting to a zero of order $4$, if we can rule out that
both  $D_2\theta_0(\zz,0) = - a_1(\zz)=0$ and $D_2\theta_1(\zz,0) = a_0(\zz)=0$. 
Together with the inclusion stated in Proposition~\ref{prop:GDS} this implies
the first statement of the proposition under the assumption
on $\theta_X$.
\par
Suppose that $\theta_X$ was zero for some $\zz \in \HH^2$, i.e.\ 
$a_0(\zz) = a_1(\zz) =0$. 
Consider the family of arithmetic
genus three curves given by the vanishing locus of $a\theta_0 + b\theta_1$ parameterized
by $(a:b) \in \proj^1$. If we blow up the four base points of $\cL$ in $A_\zz$ we obtain
a fibered surface with Euler number $-4$. If all the fibers were smooth, the formula
for a genus three fiber bundle over a projective line gives Euler number $-8$, 
a contradiction. The possible singular fibers of a section of $\cL$ are determined
in \cite{barth12pol}, see also \cite[Exercise 10.8.(1)]{bl}. 
The first possibility is a genus two curve with one
node, necessarily disjoint from the base points, the other two possibilities
consist of configurations of elliptic curves. (They can occur only on
some modular curves in $X_{D,(1,2)}$, but we will not use this.) Since $a_0=a_1=0$, 
the holomorphic one-form $du_1$ restricted to any global section of $\cL$
has a zero of order at least two at zero. This already rules out all the configurations
of elliptic curves.
\par
Still assuming that $a_0 = a_1 =0$, we consider the fibered surface 
$f: \cX = {\rm Blowup}_{4 \,\text{ points}} (A_\zz) \to \proj^1$ all whose singular fibers
are of geometric genus two. By Lemma~\ref{le:cLclass} all the fibers admit an
involution $\rho$ induced by $(-1)$ with (generically) $4$ fixed points. The
quotient is thus a curve of arithmetic genus one. Since for the singular fibers
the $4$ base points are disjoint from the node, the arithmetic genus one
curve is smooth if and only if the corresponding fiber of $f$ is smooth.
We claim that this implies that $f$ is a pullback of a \Teichmuller curve generated 
by a square-tiled surface, whose family of Jacobians has a two-dimensional fixed part, the
abelian surface $A_\zz$. In fact, consider the image of the moduli map $\proj^1 \to \moduli[3]$.
The image is embedded in $\moduli[3]$, so its tangent map, the Kodaira-Spencer
map, vanishes nowhere. Since the $2$-dimensional abelian subvariety $A_\zz$ of the
family of Jacobians is constant, this implies that the Kodaira-Spencer map
of the quotient family of elliptic curves $\cX/\langle \rho \rangle$ never vanishes.
Together with the statement on singularities the hypothesis for the characterization
of \Teichmuller curves \cite[Theorem~1.2]{bouwmoel} are met.
\par
The claim implies that this fibered surface also defines a Shimura curve and 
by \cite[Lemma~4.5]{moellerST}, the singular fibers of such a family
cannot be of geometric genus two, more precisely, the fibered surface $f$ has to be the 
unique such curve in $\moduli[3]$, described in detail in \cite[Section~3]{moellerST} or in 
\cite{HeSc}. This contradiction concludes the proof that $\theta_X$ is nowhere zero.
\par
For the second statement we look at the periods of the eigenform $\omega$ with
a $4$-fold zero using our conventions \eqref{eq:lazz}. The
periods of the first eigenform are, by definition, $\fra \oplus \ord_D^\vee z$
for some $z \in \HH$ (as opposed to $\fra^\sigma  \oplus \ord_D^\vee z$ for
the second eigenform). If the two components $W_D(4)^1$ and $W_D(4)^2$
coincided at some point the two eigenforms would lie on the same 
abelian surface with real multiplication and $(1,2)$-polarization. 
We thus would obtain an $\ord_D$-linear isomorphism 
$$\fra \oplus \ord_D^\vee \cong \fra^\sigma \oplus \ord_D^\vee.$$
Taking determinants of both sides we obtain $\fra \cong \fra^\sigma$, 
contradicting $D \equiv 1 \,(8)$.
\end{proof}
\par
The last argument is given in coordinates on flat surfaces explicitly 
in \cite[Lemma~6.2]{manhlann}, and serves for the same purpose
of distinguishing the two components.
\par
We can now collect all the information and prove all the theorems stated
in the introduction as well as at the beginning of this section.
\par
\begin{proof}[Proof of Theorem~\ref{thm:introGDX}]
Proposition~\ref{prop:GD} and \ref{prop:g4torelli} prove the first statement of the theorem.
For the second statement, note that the proof that $W_D(6)$ is a \Teichmuller curve
uses the fact that universal covering of $\HH \to \teich_4$ of
$C \to \moduli[4]$ composed with the Torelli map $\teich \to \HH_4$ to the Siegel upper
half space can be composed with a projection $\HH_4 \to \HH$ so that the composition is a
M\"obius transformation, hence a Kobayashi isometry. The non-expansion property
of the Kobayashi metric implies that $\HH \to \teich_4$ is a Kobayashi curve.
Since $\HH_4 \to \HH$ was constructed using the periods of the eigenform with
a $6$-fold zero, the composition
$$\HH \to \teich_4 \to \HH_4 \to \HH$$
factors through the universal covering map $\HH \to \HH^2$ of $W_D^X \to X_D$.
By the same argument, this is a Kobayashi curve.
\end{proof}
\begin{proof}[Proof of Theorem~\ref{thm:eulerX}]
Using the definition of the line bundles $\omega_i$ along with \eqref{eq:OMXD}, 
the class of the vanishing locus of a modular form is
\be \label{eq:WDPDcomp} [W_D^X] = [\omega_1] + 7[\omega_2]. \ee
Since $[W_D^X]$ is a Kobayashi curve with the isometric embedding
given by the first variable, we have $-[W_D^X][\omega_1] = \chi(W_D^X)$.
Thus, pairing \eqref{eq:WDPDcomp} with 
$-[\omega_1]$ and using that $[\omega_1][\omega_2] = \chi(X_D)$ on a Hilbert modular surface
gives the desired result.
\end{proof}
\par
The {\em proof of Theorem~\ref{thm:eulerS}} follows from the same intersection argument.
The {\em proof of Theorem~\ref{thm:intoXS}} is now an immediate consequence of Theorem~\ref{thm:eulerX} together
with Proposition~\ref{prop:g4torelli} and Theorem~\ref{thm:eulerS} together
with Proposition~\ref{prop:g3torelli}.

\section{An invariant and possible generalizations} \label{sec:general}

Aiming to construct more, essentially different Kobayashi curves one can use
the procedure involving theta functions to construct Hilbert modular forms of
non-parallel weight generalizing the preceding construction. Of course we
thus leave the world of \Teichmuller curves. We propose replace the integer $2$
(type of the polarization) in the genus four discussion by an arbitrary $N \in {\mathbb N}$.
First, we start with the definition of an invariant.
\par
Suppose that $C \to X_D$ is a Kobayashi curve and let $\overline{C} \to \overline{X_D}$
the closure in a good compactification of $X_D$. We define
$$ \lambda_2(C) = \frac{[\omega_1]\cdot [\ol{C}]}{[\omega_2]\cdot [\ol{C}]}$$
and call this ratio the {\em second Lyapunov exponent} of the Kobayashi curve. 
Although we defined the intersection number on a compactification, the value of $\lambda_2$
is independent of the choice of a compactification since $[\omega_i]\cdot [B_j]$ for $i=1,2$
and all components $B_j$ of the boundary divisor $B$. 
Justification for the terminology, i.e.\ the relation to a Lyapunov exponent for
the $\SL_2(\reals)$-action is given in \cite{weiss}, implicitly also in the last
section of \cite{moelPCMI}.
\par
The Propositions~\ref{prop:GD}, \ref{prop:GDS}, \ref{prop:g4torelli} and \ref{prop:g3torelli}  
can in this language be summarized as follows.
\par
\begin{prop}
Each of the Hilbert modular surfaces $X_D$ contains Kobayashi curves
of second Lyapunov exponent $1$, $1/3$ and $1/7$. The Hilbert modular surface $X_5$ moreover
contains Kobayashi curves of second Lyapunov exponent $1/2$. 
\end{prop}
\par
\begin{proof}
Hirzebruch-Zagier cycles gives $\lambda_2=1$, the curve $W_D$ give $\lambda_2 = 1/3$, 
the $W_D^X$ are curves with $\lambda_2 = 1/7$. The decagon generates a \Teichmuller
curve with $\lambda_2 = 1/2$ in $X_5$ (\cite{mcmullentor}).
\end{proof}
\par
\begin{problem}
What is the set of second Lyapunov exponents for Kobayashi curves in $X_D$?
\end{problem}
\par
The same question can be formulated for Hilbert modular surfaces of other
genera, e.g.\ in $X_{D,(1,2)}$  the curves $W_D^S$ have $\lambda_2= 1/5$. 
\par
There is an obvious generalization of the construction of $G_D^X$. 
The (Hilbert) theta functions $\theta(c_1,0)(\zz,\uu)$ for $c_1 \in \frac1N \zed^2/\zed^2$ are a basis 
of $\cL^{\otimes N}$ on the abelian variety $A_\zz$. Numbering elements in $\frac1N \zed^2/\zed^2$
by $0,\ldots,N^2-1$ we obtain theta functions $\theta_0,\ldots,\theta_{N^2-1}$. We let
$$G_D^{[N]}(\zz) = \left| \begin{matrix} \theta_0(\zz) & \theta_1(\zz) & \cdots &\theta_{N^2-1}(\zz) \\ 
\theta'_0(\zz) & \theta_1'(\zz) & \cdots &\theta'_{N^2-1}(\zz) \\
\vdots & \vdots && \vdots \\
\theta^{(N^2-1)}_0(\zz) & \theta^{(N^2-1)}_1(\zz) &\cdots &\theta^{(N^2-1)}_{N^2-1}(\zz) \\
\end{matrix} \right|. $$
As in the proof of Proposition~\ref{prop:GD} we conclude that 
$G_D^{[N]}$ is a modular form of weight $(\frac{N^2}2,\frac{N^2}2(2N^2-1))$.
\begin{problem}
Is the zero locus of $G_D^{[N]}$ irreducible? Are its components Koba\-ya\-shi 
curves in $X_D$?
\end{problem}
\par
A positive answer to this problem would at least show that the set of second
Lyapunov exponents for Kobayashi curves in $X_D$ is infinite, containing the values $\lambda_2 =\frac1{2N^2-1}$.

\bibliography{my}

\end{document}